\documentclass[final]{siamltex}

\usepackage{epsfig,amsfonts,amsmath,amssymb}
\usepackage[notref,notcite]{showkeys}
\usepackage[IT]{subfigure}

\newtheorem{thm}{Theorem}{\bf}{\it}
\newtheorem{lem}[thm]{Lemma}{\bf}{\it}
\newtheorem{ex}[thm]{Example}{\bf}{\rm}

\newcommand{\D}{\mathcal{D}}
\newcommand{\HH}{\mathcal{H}}
\newcommand{\ee}{\mathrm{e}}
\newcommand{\e}[1]{\cdot 10^{#1}}

\def\norm#1{\lVert#1\rVert}
\def\sfrac#1#2{\mbox{\footnotesize$\displaystyle\frac{#1}{#2}$}}

\title{A convergence analysis of the Peaceman--Rachford scheme for semilinear evolution equations}
\author{Eskil Hansen\thanks{Centre for Mathematical Sciences, Lund University, P.O.\ Box 118, SE-221 00 Lund, Sweden ({\tt eskil@maths.lth.se}).
        The work of the first author was supported by the Swedish Research Council under grant 621-2011-5588.}
\and Erik Henningsson\thanks{Centre for Mathematical Sciences, Lund University, P.O.\ Box 118, SE-221 00 Lund, Sweden ({\tt erikh@maths.lth.se}).
        The work of the second author was supported by the Crafoord Foundation under grant 20110535.}
}
\begin{document}
\maketitle

\begin{abstract}
The Peaceman--Rachford scheme is a commonly used splitting method for discretizing semilinear evolution equations, where the vector fields are given by the sum of one linear and one nonlinear dissipative operator. Typical examples of such equations are reaction-diffusion systems and the damped wave equation. In this paper we conduct a convergence analysis for the Peaceman--Rachford scheme in the setting of dissipative evolution equations on Hilbert spaces. We do not assume Lipschitz continuity of the nonlinearity, as previously done in the literature. First or second order convergence is derived, depending on the regularity of the solution, and a shortened proof for $o(1)$-convergence is given when only a mild solution exits. The analysis is also extended to the Lie scheme in a Banach space framework. The convergence results are illustrated by numerical experiments for Caginalp's solidification model and the Gray--Scott pattern formation problem.
\end{abstract}

\begin{keywords}
Peaceman--Rachford scheme, convergence order, semilinear evolution equations, reaction-diffusion systems, dissipative operators
\end{keywords}

\begin{AMS}
65J08, 65M12, 47H06
\end{AMS}

\pagestyle{myheadings}
\thispagestyle{plain}
\markboth{E.\ HANSEN AND E.\ HENNINGSSON}
{\sc{Convergence of the Peaceman--Rachford scheme}}

\section{Introduction}\label{sec:intro}

Semilinear evolution equations, i.e.,
\begin{equation}\label{eq:evolv}
\dot{u}=(A+F)u,\quad u(0)=\eta,
\end{equation}
are frequently encountered in biology, chemistry and physics, as they describe reaction-diffusion systems, as well as the damped wave equation. The operator $A$ is assumed to be linear, typically describing the diffusion process, and the operator $F$ can be nonlinear, e.g., arising from chemical reactions governed by the rate law. Both operators are assumed to be dissipative and may therefore give rise to stiff ODE systems when discretized in space. A common choice of temporal discretization is the (potentially) second order Peaceman--Rachford scheme. The solution at time $t=nh>0$ is then approximated as $u(nh)\approx S^n\eta$, where a single time step is given by the nonlinear operator
\begin{equation}\label{eq:PR}
S=(I-\sfrac h2F)^{-1}(I+\sfrac h2A)(I-\sfrac h2A)^{-1}(I+\sfrac h2F).
\end{equation}
As for any splitting method, the Peaceman--Rachford scheme has the advantage that the actions of the operators $A$ and $F$ are separated. This may reduce the computational cost dramatically. For example, in the context of a reaction-diffusion system the action of the linear resolvent $(I-h/2\,A)^{-1}$ can be approximated by a standard fast elliptic equation solver and the action of the nonlinear resolvent $(I-h/2\,F)^{-1}$ can often be expressed in a closed form. Further beneficial features of the scheme are that it, contrary to exponential schemes, does not require the exact flows related to $A$ and $F$, and the computational cost of evaluating the action of the operator $S$ is similar to that of the first order Lie scheme, where a time step is given by the operator
\begin{equation*}
(I-hF)^{-1}(I-hA)^{-1}.
\end{equation*}

The Peaceman--Rachford scheme was originally introduced in \cite{PEA}, with the motivation to conduct dimension splitting for the heat equation, i.e.,
\begin{equation*}
A+F=\partial_{xx}+\partial_{yy}
\end{equation*}
in two dimensions, and their approach of splitting has become a very active field of research. For an introductional reading on splitting schemes and their applications we refer to \cite[Chapter~IV]{HUN2}. Second order convergence of the scheme has been proven in \cite{DES}, when applied to reaction-diffusion systems given on the whole $\mathbb{R}^d$ and with Lipschitz continuous nonlinearities $F$. Second order convergence has also been established in~\cite{HAN1} under the assumption that the operator $F$ is linear and unbounded, i.e., applicable to dimension splitting of linear parabolic equations. Convergence, but without any order, is proven for the fully nonlinear problem in \cite{LIO}. See also \cite{HUN1,SCH} for further numerical considerations. In the setting of exponential splitting schemes, second order convergence for the Strang splitting has been established for semilinear problems in \cite{FAO,HAN2}. Further results for exponential schemes are, e.g., surveyed in \cite[Section III.3]{LUB} and \cite{MCL}.

However, to the best of our knowledge, there is still no convergence (order) analysis of the Peaceman--Rachford scheme applied to equation \eqref{eq:evolv}, which does not assume that the dissipative operator $F$ is either linear or Lipschitz continuous. Note that the latter assumption is rather restrictive. A concrete reaction-diffusion problem that does not fulfill the assumption is the Allen--Cahn equation, where
 \begin{equation*}
(A+F)u=\Delta u+(u-u^3),
\end{equation*}
equipped with suitable boundary conditions and interpreted as an evolution equation on $L^2(\Omega)$. The aim of this paper is therefore to conduct a convergence analysis for the scheme at hand without assuming linearity or Lipschitz continuity of the operator~$F$.

\section{Problem setting}\label{sec:set}

Let $\HH$ be a real Hilbert space with the inner product and the norm denoted as  $(\cdot,\cdot)$ and $\norm{\cdot}$, respectively.
For every operator $G:\D(G)\subseteq\HH\to\HH$ we define its Lipschitz constant $L[G]$ as the smallest possible constant $L\in[0,\infty]$ such that
\begin{equation*}
\norm{Gu-Gv}\leq L\norm{u-v}\qquad\text{for all $u$, $v$ in $\D(G)$}.
\end{equation*}
Furthermore, an operator $G:\D(G)\subseteq\HH\to\HH$ is maximal (shift) dissipative if and only if there is a constant $M[G]\geq 0$ for which the
operator $G$ satisfies the range condition
\begin{equation}\label{eq:rang}
\mathcal{R}(I-hG)=\HH\qquad\text{for all $h>0$ such that $hM[G]<1$},
\end{equation}
and the dissipativity condition
\begin{equation}\label{eq:diss}
(Gu-Gv,u-v)\leq M[G]\norm{u-v}^2\qquad\text{for all $u$, $v$ in $\D(G)$}.
\end{equation}
A direct consequence of an operator $G$ being maximal dissipative is that the related resolvent $(I-hG)^{-1}:\HH\to\D(G)\subseteq\HH$ is well defined and
\begin{equation*}
L[(I-hG)^{-1}]\leq 1/(1-hM[G])
\end{equation*}
for all $h>0$ such that $hM[G]<1$. With this in place, we can characterize our problem class as follows:

{\sc Assumption 1.}
\emph{ The operators $A:\D(A)\subseteq\HH\to\HH$,
$F:\D(F)\subseteq\HH\to\HH$ and $A+F:\D(A)\cap\D(F)\subseteq\HH\to\HH$
are all maximal dissipative on $\HH$. }

If Assumption 1 is valid, then there exists a unique mild solution $u$ to the semilinear evolution equation \eqref{eq:evolv} for every $\eta$ in the closure of
$\D(A)\cap\D(F)$. The related solution operator is given by a nonlinear semigroup $\{\ee^{t(A+F)}\}_{t\geq 0}$, where
\begin{equation*}
u(t)=\ee^{t(A+F)}\eta.
\end{equation*}
The nonlinear operator $\ee^{t(A+F)}$ is invariant over the closure of
$\D(A)\cap\D(F)$ and can be characterized by the limit
\begin{equation*}
\ee^{t(A+F)}\eta =\lim_{n\to\infty}\bigl(I-\sfrac tn (A+F)\bigr)^{-n}\eta.
\end{equation*}
A contemporary survey of maximal dissipative operators and nonlinear semigroups can be found in the monograph~\cite[Sections~3.1 and~4.1]{BAR}.

\section{Preliminaries}\label{sec:prel}

In order to shorten the notation slightly, we introduce the abbreviations
\begin{equation*}
a=\sfrac 12 h A,\quad \alpha=(I-a)^{-1}, \quad f=\sfrac 12 hF\quad\text{and}
\quad \varphi=(I-f)^{-1}.
\end{equation*}
We will also make frequent use of the identities
\begin{equation*}
I=\alpha-\alpha a=\alpha-a\alpha\quad\text{and}\quad I=\varphi(I-f)=\varphi-f\varphi,
\end{equation*}
without further references. The time stepping operator of the Peaceman--Rachford scheme \eqref{eq:PR} then reads as
\begin{equation}\label{eq:PR2}
S=\varphi(I+a)\alpha(I+f).
\end{equation}

Due to the presence of the term $I+f$ in \eqref{eq:PR2}, the time stepping operator $S$ is, in general, not Lipschitz continuous. Hence, one needs to
modify the scheme in order to establish stability and convergence. To this end, we consider the auxiliary time stepping operator
\begin{equation}\label{eq:aux}
R=(I+a)\alpha(I+f)\varphi,
\end{equation}
which relates to $S$ via the equality
\begin{equation*}
S^j\varphi=\varphi R^j
\end{equation*}
for all $j\geq 0$.

\begin{lem}\label{lem:stab}
If Assumption 1 is valid and $h\max\{M[A],M[F]\}\leq 1$, then
\begin{equation*}
L[R^j]\leq\ee^{3/2\, jh(M[A]+M[F])}
\end{equation*}
for every $j\geq 0$.
\end{lem}

\begin{proof}
Let $u,v$ be two arbitrary elements of $\D(F)$. A twofold usage of the dissipativity then gives the inequality
\begin{equation*}
\begin{split}
\norm{(I+f)u-(I+f)v}^2 & = \norm{u-v}^2+2(fu-fv,u-v)+\norm{fu-fv}^2\\
 					& \leq (1+hM[F])\norm{u-v}^2+\norm{fu-fv}^2\\
  					& \leq (1+2hM[F])\norm{u-v}^2-2(fu-fv,u-v)+\norm{fu-fv}^2\\
  					& = 2hM[F]\norm{u-v}^2+\norm{(I-f)u-(I-f)v}^2.
\end{split}
\end{equation*}
Replacing $u,v$ by $\varphi z,\varphi w$,  then yields that
\begin{equation*}
\begin{split}
\norm{(I+f)\varphi z-(I+f)\varphi w} & \leq (2hM[F] L[\varphi]^2+1)^{1/2}\norm{z-w}\\
							  & \leq (1+h/2\,M[F])/(1-h/2\,M[F]) \norm{z-w}.						
\end{split}
\end{equation*}
As the above inequality is valid for any $z,w\in\HH$, we obtain the bound
\begin{equation*}
L[(I+f)\varphi]\leq (1+h/2\,M[F])/(1-h/2\,M[F]).
\end{equation*}
The same type of Lipschitz continuity holds for the operator $(I+a)\alpha$, hence,
\begin{equation*}
L[R^j]\leq L[(I+a)\alpha]^jL[(I+f)\varphi]^j\leq\ee^{3/2\, jh(M[A]+M[F])},
\end{equation*}
where the last inequality follows as $(1+x)/(1-x)\leq \ee^{x+2x}$ for all
$x\in [0,1/2]$. \qquad\end{proof}

\section{Convergence of the Peaceman--Rachford scheme}\label{sec:PR}

The scheme is often employed for problems with rather smooth solutions, e.g., in the context of reaction-diffusion equations, and we therefore start to
derive a global error bound valid for sufficiently regular solutions.

{\sc Assumption 2.}
\emph{ The evolution equation \eqref{eq:evolv} has a classical solution $u$, i.e., the function $u\in C^1([0,T];\HH)$ satisfies $\dot{u}(t)=(A+F)u(t)$ for every
time $t\in[0,T]$. Furthermore, the solution satisfies one of the following statements:
\begin{romannum}
\item $u\in W^{2,1}(0,T;\HH)$ and $A\dot{u}\in L^1(0,T;\HH)$;
\item $u\in W^{3,1}(0,T;\HH)$, $A\ddot{u}\in L^1(0,T;\HH)$, and $A^2\dot{u}\in L^1(0,T;\HH)$.
\end{romannum} }

With such regularity present the Peaceman--Rachford scheme is either first or second order convergent.

\begin{thm}\label{thm:conv12}
Consider the Peaceman--Rachford discretization \eqref{eq:PR} of the semilinear evolution equation \eqref{eq:evolv}. If Assumption 1 is valid and
$h\max\{M[A],M[F]\}\leq 1$, then the global error of the Peaceman--Rachford approximation can be bounded as
\begin{equation*}
\norm{u(nh)-S^n\eta}\leq \sfrac 52\, h^p\,\ee^{3/2\, T(M[A]+M[F])}
\sum_{\ell=0}^p\norm{A^{p-\ell}u^{(\ell+1)}}_{L^1(0,T;\HH)},\quad nh\leq T,
\end{equation*}
where $p=1$ under Assumption 2.i and $p=2$ under Assumption 2.ii.
\end{thm}

\begin{proof}
First, we expand the global error as the telescopic sum
\begin{equation*}
\begin{split}
u(nh)-S^n\eta& =\sum_{j=1}^n S^{n-j}u(t_j)-S^{n-j+1}u(t_{j-1})\\
		     & =\sum_{j=1}^n \varphi R^{n-j}(I-f)u(t_j)-\varphi R^{n-j}(I-f)Su(t_{j-1}),
\end{split}
\end{equation*}
where $t_j=jh$. This expansion together with Lemma \ref{lem:stab} yields that
\begin{equation}\label{eq:error}
\begin{split}
\norm{u(nh)-S^n\eta}& \leq\sum_{j=1}^n  L[\varphi] L[R^{n-j}]\norm{(I-f)u(t_j)-(I-f)Su(t_{j-1})}\\
				& \leq\ee^{t_nM} \sum_{j=1}^n\norm{(I-f)u(t_j)-(I-f)Su(t_{j-1})},
\end{split}
\end{equation}
where $M=3/2\,(M[A]+M[F])$. Next, we seek a suitable representation of the difference
\begin{equation*}
d_j=(I-f)u(t_j)-(I-f)Su(t_{j-1}).
\end{equation*}
The second term can be written as
\begin{equation*}
(I-f)Su(t_{j-1})=\alpha\bigl(I+(a+f)\bigr)u(t_{j-1})+a\alpha fu(t_{j-1}).
\end{equation*}
In order to match the first term $(I-f)u(t_j)$ with the above expression we expand the identity in terms of $a$ and $\alpha$, and obtain that
\begin{equation}\label{eq:expan}
\begin{split}
(I-f)u(t_j)& =(\alpha-\alpha a-f)u(t_j)\\
	     & =\alpha\bigl(I-(a+f)\bigr)u(t_j)+(\alpha-I) f u(t_j)\\
	     & =\alpha\bigl(I-(a+f)\bigr)u(t_j)+a\alpha f u(t_j).
\end{split}
\end{equation}
This gives us the representation $d_j=\alpha q_j +s_j$, where
\begin{equation*}
\begin{split}
q_j& =u(t_j)-u(t_{j-1})-(a+f)u(t_j)-(a+f)u(t_{j-1})\qquad\text{and}\\
s_j& =a\alpha\bigl(fu(t_j)-fu(t_{j-1})\bigr).
\end{split}
\end{equation*}
By Assumption 2, the solution $u$ is an element in $W^{\ell,1}(0,T;\HH)$, with
$\ell=2$ or $3$,  and the $q_j$ term can therefore be written as
\begin{equation*}
\begin{split}
q_j& =\int_{t_{j-1}}^{t_j}\dot{u}(t)\,\mathrm{d}t-\sfrac 12h\bigl(\dot{u}(t_j)+\dot{u}(t_{j-1})\bigr)\\
     & =h\int_{t_{j-1}}^{t_j}(\sfrac12-\sfrac{t-t_{j-1}}{h})\ddot{u}(t)\,\mathrm{d}t\\
     & =h^2\int_{t_{j-1}}^{t_j}\sfrac12\sfrac{t-t_{j-1}}{h}\sfrac{t-t_j}{h}u^{(3)}(t)\,\mathrm{d}t.
\end{split}
\end{equation*}
Hence, the term $q_j$ is simply the local error of the trapezoidal rule and can either be expressed as a first or a second order term in $h$, depending on
the regularity of the solution~$u$. Furthermore, the splitting error $s_j$ can also be interpreted as a first or a second order term in $h$. This follows as
\begin{equation*}
\begin{split}
s_j& =a\alpha\bigl(\sfrac12 h[\dot{u}(t_j)-\dot{u}(t_{j-1})]-a[u(t_j)-u(t_{j-1})]\bigr)\\
     & =\sfrac12 h a\alpha\bigl(\int_{t_{j-1}}^{t_j}\ddot{u}(t)\,\mathrm{d}t-A\int_{t_{j-1}}^{t_j}\dot{u}(t)\,\mathrm{d}t\bigr)\\
     & =\sfrac14 h^2 \alpha\bigl(A\int_{t_{j-1}}^{t_j}\ddot{u}(t)\,\mathrm{d}t-A^2\int_{t_{j-1}}^{t_j}\dot{u}(t)\,\mathrm{d}t\bigr).
\end{split}
\end{equation*}
Note that the operators $A^\ell$ can be interchanged with the integrations, as $A$ is closed (via Assumption 1) and the integrated functions are assumed to
be sufficiently regular (Assumption 2).

Finally, the above representations of the terms $q_j$ and $s_j$ together with the observations that $L[\alpha]\leq 1/(1-1/2\,hM[A])\leq2$ and
$L[a\alpha]=L[\alpha-I]\leq 3$ when $hM[A]\leq 1$ give us the bound
\begin{equation*}
\norm{d_j} \leq \sfrac 52\,h^p\,\int_{t_{j-1}}^{t_j} \sum_{\ell=0}^p\norm{A^{p-\ell}u^{(\ell+1)}(t)}\,\mathrm{d}t,
\end{equation*}
with $p=1$ or $2$. Combining \eqref{eq:error} with the above inequality yields the sought after error bound.
\qquad\end{proof}

If the regularity prescribed in Assumption 2 is not present one can still obtain convergence of the Peaceman--Rachford approximation to the mild solution. This follows by a Lax-type theorem due to Br\'ezis and Pazy \cite{BRE}. Note that the $o(1)$-convergence of the scheme is given in \cite[Theorem 2]{LIO}, when $M[A]$ and $M[F]$ are equal to zero. However, in the current notation we are able to give a significantly shorter proof.
\begin{thm}\label{thm:conv0}
Consider the mild solution $u$ of the semilinear evolution equation~\eqref{eq:evolv} and its approximation $S^n\eta$ by the Peaceman--Rachford scheme
\eqref{eq:PR}. If Assumption~1 is valid and $\D(A)\cap\D(F)$ is dense in $\HH$, then
\begin{equation*}
\lim_{n\to\infty} S^n\eta = u(t),
\end{equation*}
for every $\eta\in\D(F)$ and $t\geq 0$.
\end{thm}

\emph{Proof.}
The operator $F$ is assumed to be maximal dissipative and densely defined, as $\overline{\D(A)\cap\D(F)}\subseteq\overline{\D(F)}$, which implies the limits
\begin{equation*}
\lim_{h\to 0} F\varphi v= Fv\quad\text{and}\quad\lim_{h\to 0} \varphi w= w,
\end{equation*}
for every $v\in\D(F)$ and $w\in \HH$; see, e.g., the proof of \cite[Proposition 11.3]{DEI}. The same type of limits obviously also hold for the operator $A$. Next, consider the auxiliary scheme \eqref{eq:aux}, which can be reformulated as
\begin{equation*}
R = (I+2a\alpha)(I+2f\varphi)=I+2a\alpha+2f\varphi+4(\alpha-I)(f\varphi-f)+4(\alpha-I)f
\end{equation*}
on $\D(F)$. Hence, the auxiliary scheme is consistent, i.e.,
\begin{equation*}
\lim_{h\to 0}\sfrac1h(R-I)v = (A+F)v,
\end{equation*}
for every $v\in \D(A)\cap\D(F)$. Moreover, Lemma \ref{lem:stab} yields stability in the sense that  $L[R]\leq 1+Ch+o(h)$, and the auxiliary approximation $R^n\eta$ is therefore convergent, for every $\eta\in\HH$, by \cite[Corollary~4.3]{BRE}.

The sought after convergence of the Peaceman--Rachford approximation is then obtained for all $\eta\in\D(F)$ via the error bound
\begin{equation*}
\norm{u(t)-S^n\eta}\leq \norm{(I-\varphi)u(t)}+L[\varphi]\norm{u(t)-R^n\eta}+L[\varphi]L[R^n]\norm{\eta-(I-f)\eta}.\qquad\endproof
\end{equation*}

\section{Convergence of the Lie scheme}\label{sec:Lie}

With the machinery of the previous sections in place, we can also derive a global error bound for the Lie scheme, where a single time step is given by the
operator
\begin{equation}\label{eq:Lie}
S=(I-hF)^{-1}(I-hA)^{-1}.
\end{equation}
For this analysis we replace Assumption 2 by the one stated below.

{\sc Assumption 3.}
\emph{The evolution equation \eqref{eq:evolv} has a classical solution $u$ such that $u\in W^{2,1}(0,T;\HH)$ and $A\dot{u}(t), A^2u(t)\in\HH$ for every
time $t\in[0,T]$.}

In this context, we obtain first order convergence for the Lie approximation:

\begin{thm}\label{thm:convLie}
Consider the Lie discretization \eqref{eq:Lie} of the semilinear evolution equation \eqref{eq:evolv}. If Assumptions 1 and 3 hold and
$h\max\{M[A],M[F]\}\leq 1/2$, then the global error of the Lie approximation can be bounded as
\begin{equation*}
\norm{u(nh)-S^n\eta}\leq 2\,h\, \ee^{2T(M[A]+M[F])}\bigl(\norm{\ddot{u}}_{L^1(0,T;\HH)}+T\norm{A\dot{u}-A^2u}_{L^\infty(0,T;\HH)}\bigr),
\end{equation*}
where $nh\leq T$.
\end{thm}

\emph{Proof.}
In order to mimic the earlier convergence proof, we introduce the abbreviations
\begin{equation*}
a=hA,\quad \alpha=(I-a)^{-1}, \quad f=hF\quad\text{and}
\quad \varphi=(I-f)^{-1}.
\end{equation*}
The Lie scheme then reads as $S=\varphi\alpha$, and its stability follows by
\begin{equation*}
L[S^j]\leq L[\varphi]^jL[\alpha]^j\leq \ee^{2jh(M[A]+M[F])}.
\end{equation*}
We again expand the global error in a telescopic sum and obtain the bound
\begin{equation}\label{eq:errorLie}
\begin{split}
\norm{u(nh)-S^n\eta}& \leq\sum_{j=1}^n  \norm{S^{n-j}u(t_j)-S^{n-j+1}u(t_{j-1})}\\
				& \leq \sum_{j=1}^n L[S^{n-j}]L[\varphi]\norm{(I-f)u(t_j)-(I-f)Su(t_{j-1})}\\
				& \leq \ee^{2t_n(M[A]+M[F])}\sum_{j=1}^n\norm{(I-f)u(t_j)-\alpha u(t_{j-1})},
\end{split}
\end{equation}
where $t_j=jh$.  With the expansion \eqref{eq:expan} of the term $(I-f)u(t_j)$, we can once more express the difference
\begin{equation*}
d_j=(I-f)u(t_j)-\alpha u(t_{j-1})=\alpha q_j+s_j
\end{equation*}
in terms of a quadrature error $q_j$ and a splitting error $s_j$, where
\begin{equation*}
\begin{split}
q_j& =\int^{t_j}_{t_{j-1}} \dot{u}(t)\,\mathrm{d}t-h\dot{u}(t_j)=h\int^{t_j}_{t_{j-1}} \sfrac{t_{j-1}-t}{h}\ddot{u}(t)\,\mathrm{d}t
           \qquad\text{and}\\[7pt]
s_j& =a\alpha fu(t_j)=h^2\alpha A\bigl(\dot{u}(t_j)-Au(t_j)\bigr).
\end{split}
\end{equation*}
The sought after error bound then follows as
\begin{equation*}
\norm{d_j}\leq 2\,h\,\bigl(\int_{t_{j-1}}^{t_j} \norm{\ddot{u}(t)}\,\mathrm{d}t+h\norm{A\dot{u}(t_j)-A^2u(t_j)}\bigr).\qquad\endproof
\end{equation*}

A somewhat surprising result is that the Peaceman--Rachford scheme requires less regularity than the Lie scheme in order to obtain first order
convergence, as no requirement is made regarding the term $A^2u(t)$; compare Assumptions 2.i and 3.

Even though the Lie scheme may have a less beneficial error structure, it has a significant advantage over most schemes, namely, it is stable even if $\HH$
is merely a Banach space and the derived global error bound is still valid in a Banach space framework. The necessary modification is to generalize the
dissipativity property \eqref{eq:diss} as follows: Let $\mathcal{X}$ be a real Banach space. A nonlinear operator $G:\D(G)\subseteq\mathcal{X}\to
\mathcal{X}$ is said to be dissipative if and only if
\begin{equation*}
[Gu-Gv,u-v]\leq M[G]\norm{u-v}^2
\end{equation*}
for every $u,v\in\D(G)$. Here,
$[\cdot,\cdot]:\mathcal{X}\times\mathcal{X}\to\mathbb{R}$ denotes the (left) semi-inner product \cite[p.~96]{DEI} defined as
\begin{equation*}
[u,v]=\norm{v}\lim_{\varepsilon\to 0^-}\sfrac 1\varepsilon(\norm{v+\varepsilon u}-\norm{v}).
\end{equation*}
With this extended definition of dissipativity, we still have that the resolvent of a maximal dissipative operator $G$ exists and
$L[(I-hG)^{-1}]\leq 1/(1-hM[G])$. Hence, the Lie scheme \eqref{eq:Lie} is well defined and Theorem \ref{thm:convLie} holds by the very same proof.

Note that the Peaceman--Rachford results of \S\ref{sec:PR} can not be extended to this Banach spaces framework, as the needed generalization of
Lemma \ref{lem:stab} is not true. The reason for this is that the terms $L[(I+f)\varphi]$ and $L[(I+a)\alpha]$ are, in general, not of the form
$1+\mathcal{O}(h)$  when the Hilbert structure is lost. A concrete example of a maximal dissipative operator $A$ on a Banach space with $M[A]=0$ and $L[(I+a)\alpha]\geq 1.2$, for all $h>0$, can be found in \cite[Appendix 2]{RAS}.

\section{Applications}\label{sec:appl}

We conclude with two examples of reaction-diffusion systems which fit into the framework of maximal dissipative operators and for which the Peaceman--Rachford scheme becomes an efficient temporal discretization.

\begin{ex}\label{ex:phase} {\rm
Consider the equation system
\begin{equation}\label{eq:phase}
\left\{\begin{aligned}
\dot{\theta}+\ell\dot{\phi}& =  \Delta \theta,\\
\dot{\phi} & = \Delta \phi +\phi-\phi^3+\theta,
\end{aligned}\right.
\end{equation}
where $\ell$ is a positive constant and the equation is given on  $\Omega\times[0,T]\subseteq \mathbb{R}^d\times \mathbb{R_+}$, $d=2,3$, and is equipped with suitable boundary and initial conditions. Equations of this form have been proposed, e.g., when modeling solidification
processes~\cite{CAG}. In the solidification model, $\theta$ represents the temperature and the continuously varying order parameter $\phi$ describes the transition of the material from the liquid phase ($\phi\approx 1$) to the solid phase ($\phi\approx -1$).

By the variable change $\psi=\theta+\ell\phi$, the system \eqref{eq:phase} can be reformulated as a semilinear evolution equation \eqref{eq:evolv}, where
$u=[\psi,\phi]^\mathrm{T}$,
\begin{equation*}
A=\left(\begin{matrix} 1 & -\ell\\ 0 & 1\end{matrix}\right)\Delta
\qquad\text{and}\qquad
Fu=\left[\begin{array}{c} 0\\ (1-\ell)\phi-\phi^3+\psi\end{array}\right].
\end{equation*}
Evaluating a time step of the Peaceman--Rachford  scheme \eqref{eq:PR} then consists of twice employing a standard solver for elliptic problems of the form $(I-h/2\,\Delta)v=w$, in order to evaluate the actions of the linear resolvent $(I-h/2\,A)^{-1}$, and the nonlinear resolvent $(I-h/2\,F)^{-1}$ can be computed analytically. Global errors and the presence of second order convergence are exemplified in Figure~\ref{fig:error}.

\begin{figure}[t]
\centering
\subfigure[]{
    \includegraphics[width=0.44\textwidth ]{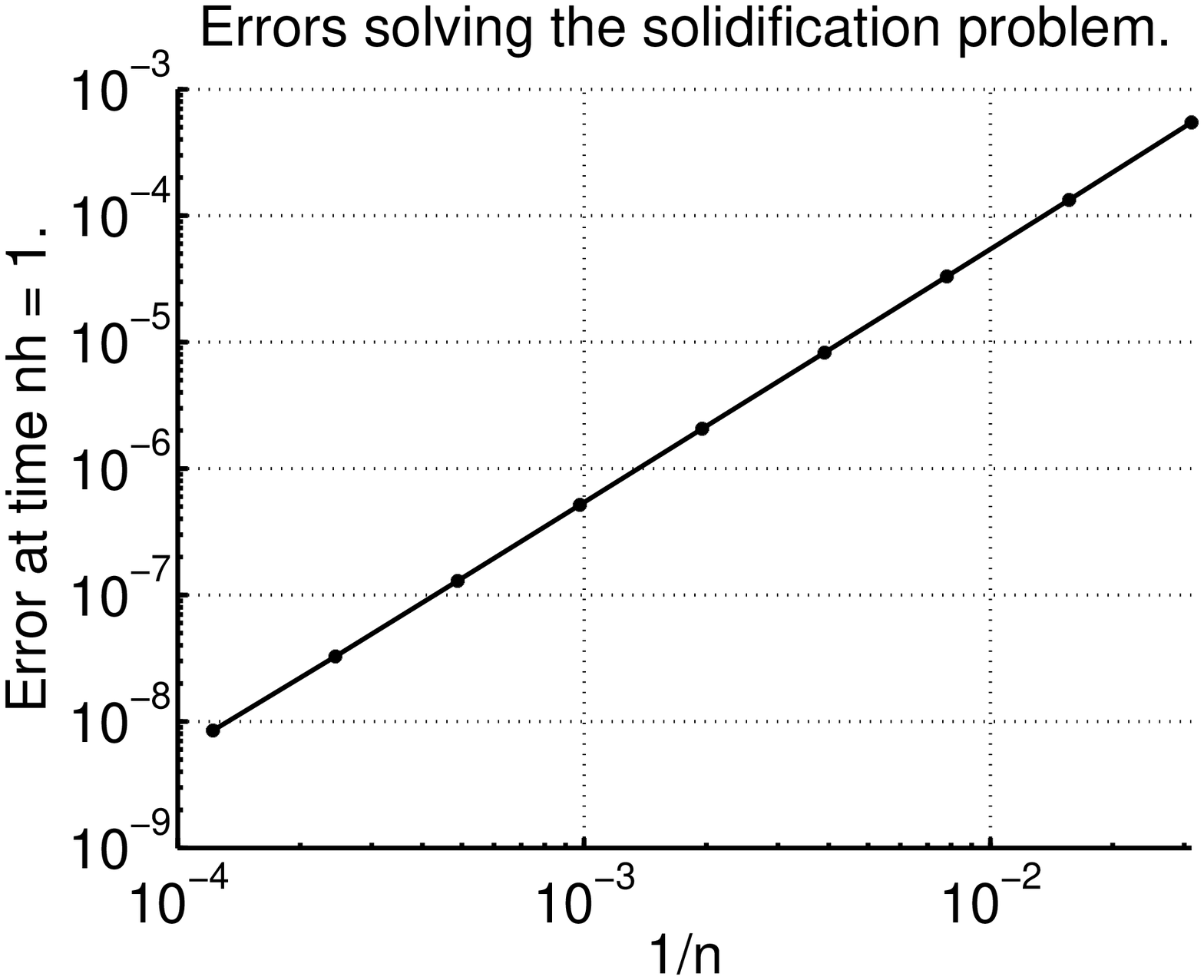}
	\label{fig:cag_conv}}
\subfigure[]{
    \includegraphics[width=0.44\textwidth ]{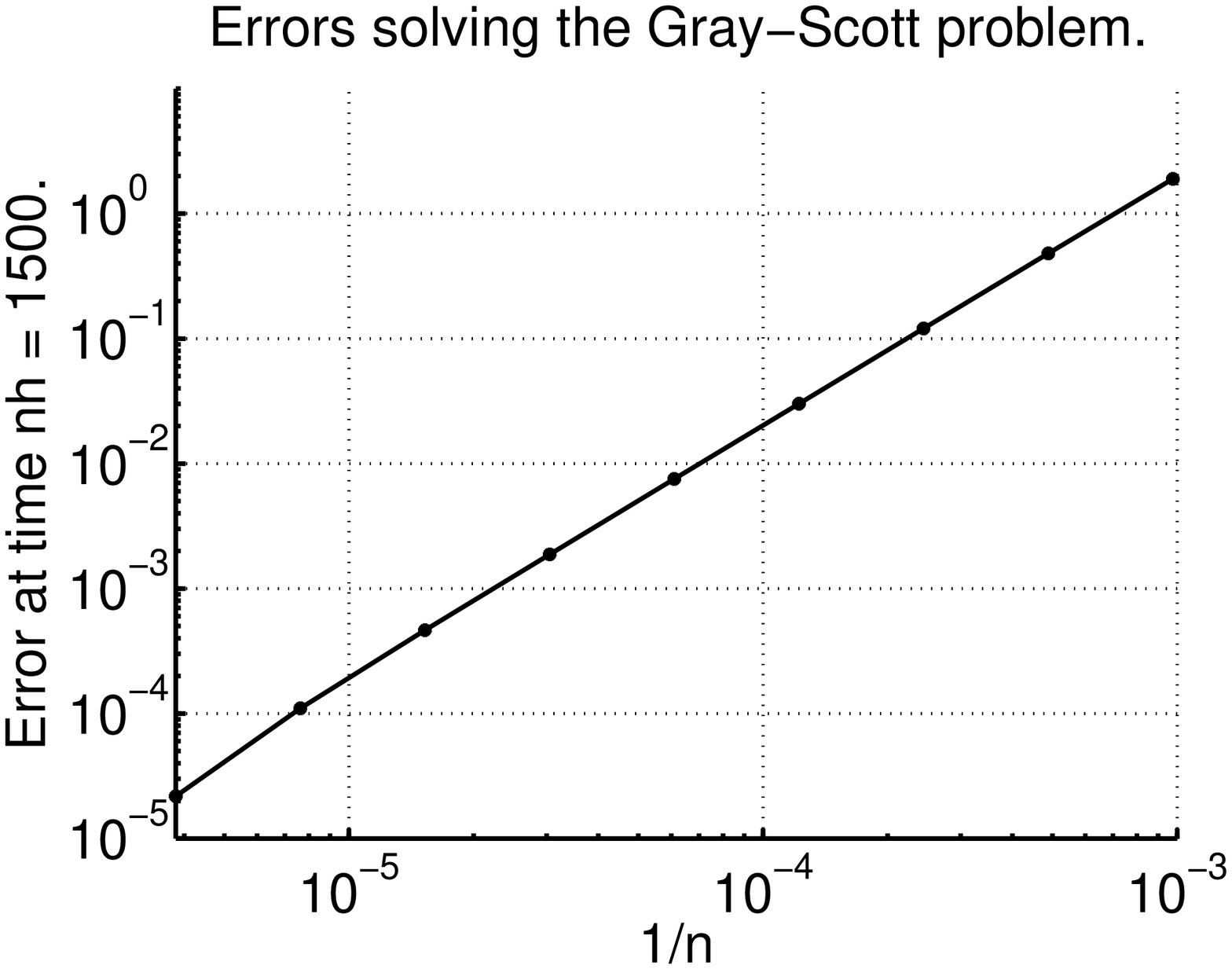}
	\label{fig:gs_conv}}
\caption{To show convergence for the Peaceman--Rachford scheme, it is applied to the two-dimensional solidification problem~\eqref{eq:phase} and the Gray--Scott pattern formation problem~\eqref{eq:gs}. In Figures \ref{fig:cag_conv} and \ref{fig:gs_conv} the global errors, at times $nh=1$ and $1500$, measured in the norms induced by the inner products \eqref{eq:cag_inner} and \eqref{eq:gs_inner}, respectively, are plotted against the inverse of the number of time steps $1/n$. Second order convergence of the scheme is observed. Both equations are equipped with periodic boundary conditions and given on the domain  $\Omega = (-\pi, \pi)^2$ with an equidistantly spaced grid  with $2^9$ nodes in each dimension. The reference solution is found by the same scheme but on a finer grid, $2^{10}$ nodes in each space dimension and $2^{19}$ time steps. Parameters used are $l=1/2$, $d_1 = 8\e{-4}$, $d_2 = 4\e{-4}$, $l_1 = 0.024$ and $l_2 = 0.084$. Initial values used for the solidification problem are $\theta_0(x) = 1$ and $\phi_0(x) = \exp{(-20(x_1^2+x_2^2/6))} + \exp{(-20(x_1^2/6+x_2^2))} - 1$. For the Gray--Scott problem we define $(u_2)_0$ as a sum of four translated ``humps'' with midpoints $y = (\pm \pi/10, \pm \pi/10)$ and radius $\epsilon = \pi/10$. The sum is scaled by a factor $\exp{(1)}/4$ and the ``hump'' is defined by $g_{y,\epsilon}(x) = \exp{(-\epsilon^2/(\epsilon^2-|x-y|^2))}$ if $|x-y| < \epsilon$, $0$ otherwise. See Figure~\ref{fig:patt} for a contour plot. The first component is defined by $(u_1)_0(x) = 1-2(u_2)_0(x)$. For both problems the actions of $(I-h/2\,A)^{-1}$ are efficiently evaluated with the help of an {\sc fft}-algorithm. The actions of the nonlinear resolvents $(I-h/2\,F)^{-1}$ can also be efficiently evaluated as they give rise to cubic equations which can be solved analytically.}
\label{fig:error}
\end{figure}

In order to interpret $A$ and $F$ as maximal dissipative operators, we choose to work in the Hilbert space $\HH=[L^2(\Omega)]^2$
equipped with the inner product
\begin{equation} \label{eq:cag_inner}
(u_1,u_2)=(\psi_1,\psi_2)_{ L^2(\Omega)}+\ell^2(\phi_1,\phi_2)_{ L^2(\Omega)}.
\end{equation}
Assume that the domain $\Omega$ has a sufficiently regular boundary and the system is equipped with, e.g., periodic boundary conditions or homogeneous Dirichlet or Neumann boundary conditions. The Laplacian $\Delta:\D(\Delta)\subseteq L^2(\Omega)\to L^2(\Omega)$ is then a maximal dissipative operator, with $M[\Delta]=0$; compare with \cite[Section~II.2]{TEM}. Here, the domain $\D(\Delta)$ can be identified as
$H^2_{per}(\Omega), H^2(\Omega)\cap H^1_{0}(\Omega)$ or $\{u\in H^2(\Omega):\gamma_1u=0\}$, when periodic, Dirichlet or Neumann conditions are imposed, respectively.

Let $\D(A)=[\D(\Delta)]^2$, then the operator $A:\D(A)\subseteq\HH\to\HH$ is maximal dissipative, with $M[A]=0$. The range condition
\eqref{eq:rang} trivially holds, as $\mathcal{R}(I-h\Delta)=L^2(\Omega)$, and the dissipativity \eqref{eq:diss} follows by the inequality
\begin{equation*}
(Au,u)\leq -\sfrac12\bigl(\norm{\nabla\psi}^2_{L^2(\Omega)}+\ell^2\norm{\nabla\phi}^2_{L^2(\Omega)}\bigr).
\end{equation*}

The operator $F:\D(F)\subseteq\HH\to\HH$ fulfills the range condition whenever its second component, which we denote by $F_2$, satisfies it on
$L^2(\Omega)$ for a fixed $\psi\in L^2(\Omega)$. This can be proven by, e.g., observing that the operator $I-h\hat{F}_2:L^4(\Omega)\to
L^4(\Omega)^*$ is surjective when $h(1-\ell)\leq 1$, where
\begin{equation*}
\hat{F}_2:\phi\mapsto\int_{\Omega} \bigl((1-\ell)\phi-\phi^3+\psi\bigr)(\cdot)\,\mathrm{d}x.
\end{equation*}
The surjectivity follows as $I-h\hat{F}_2$ fulfills the hypotheses of the Browder--Minty theorem \cite[Theorem~26.A]{ZEI}. The operator $F_2(\psi,\cdot)$
can then be identified as the restriction of  $\hat{F}_2$ to the set
\begin{equation*}
\{\phi\in L^4(\Omega):\hat{F}_2\phi\in L^2(\Omega)^*\}=L^6(\Omega),
\end{equation*}
i.e., $\D(F)=L^2(\Omega)\times L^6(\Omega)$, and the range condition then holds for $F$ on $\mathcal{\HH}$ by construction. Finally, $F$ is also
dissipative, as
\begin{equation*}
(Fu_1-Fu_2,u_1-u_2) \leq (\sfrac 32-\ell)\ell^2\norm{\phi_1-\phi_2}^2_{L^2(\Omega)}+\ell^2\norm{\psi_1-\psi_2}^2_{L^2(\Omega)}
                                          -\sfrac14\ell^2\norm{\phi_1-\phi_2}^4_{L^4(\Omega)}.
\end{equation*}
Hence, the operator $F$ is maximal dissipative, with $M[F]\leq\max\{3/2-\ell,\ell^2\}$.

The maximal dissipativity of $A+F$ follows by employing a standard perturbation result, as  done in the proof of \cite[Theorem~5.5]{BAR}.
Existence of a classical solution to \eqref{eq:phase} and further regularity results can be found \cite[Section~3]{CAG}, in the context of Dirichlet boundary
conditions.
}
\end{ex}

\begin{ex}\label{ex:react} {\rm
In the previous example the polynomial nonlinearity could be interpreted as a dissipative operator, due to the presence of the term $-\phi^3$.
However, even if this dissipative structure is not present one can still fit polynomial nonlinearities into the framework of maximal dissipative operators, by requiring further regularity and boundary condition compatibility.

Assume that the operator $A:\D(A)\subseteq\HH\to\mathcal\HH$ is maximal dissipative, and therefore also closed. The idea is to replace the Hilbert space $\HH$ by the domain $\D(A)$ which is again a Hilbert space, when equipped with the graph inner product
\begin{equation}\label{eq:gs_inner}
(u,v)_A=(Au,Av)+(u,v).
\end{equation}
The  operator $A:\D(A^2)\subseteq\D(A)\to\D(A)$ is still maximal dissipative, with the same constant $M[A]$. If the domain $\D(A)$ is a Banach algebra, then any polynomial nonlinearity $F$, with $F(0)=0$ in case $\D(A)$ lacks an identity element, maps $\D(A)$ into itself and is locally Lipschitz continuous on $\D(A)$, i.e., for every $r>0$ there exists an $L_r[F]\in[0,\infty)$ such that
\begin{equation*}
\norm{Fu-Fv}_A\leq L_r[F]\norm{u-v}_A\quad\text{for all}\quad u,v\in B_r=\{u\in\D(A):\norm{u}_A\leq r\}.
\end{equation*}
One can then introduce the truncation
\begin{equation*}
F_ru=\begin{cases}
Fu& \text{if}\quad u\in B_r,\\
F\bigl(r/\norm{u}_A\,u\bigr)& \text{if}\quad u\in \D(A)\backslash B_r,
\end{cases}
\end{equation*}
and the new operator  $F_r:\D(A)\to\D(A)$ is (globally) Lipschitz continuous. Hence, both $F_r$ and $A+F_r:\D(A^2)\subseteq\D(A)\to\D(A)$ are maximal dissipative, and the convergence results of \S\ref{sec:PR} and \S\ref{sec:Lie} are valid for all time intervals $[0,T(\eta)]$ for which the exact solution remains in $B_r$.

As an example, consider the evolution equation \eqref{eq:evolv}, with
\begin{equation*}
A=\left(\begin{matrix} d_1 & &\\  & \ddots &\\ & & d_s\end{matrix}\right)\Delta
\qquad\text{and}\qquad
Fu=\left[\begin{array}{c} p_1(u_1,\ldots,u_s)\\ \vdots \\ p_s(u_1,\ldots,u_s) \end{array}\right],
\end{equation*}
where $d_i$ are positive constants and $p_i$ are polynomials with $s$-arguments. For periodic boundary conditions and $\HH=[L^2(\Omega)]^s$ the operator $A:\D(A)\subseteq\HH\to\mathcal\HH$ is maximal dissipative when defined on the Banach algebra $\D(A)=[H^2_{per}(\Omega)]^s$.
For the Gray--Scott pattern formation model \cite[p.~21]{HUN2}, where
\begin{equation}\label{eq:gs}
p_1(u_1,u_2)=-u_1u_2^2+\ell_1(1-u_1)\qquad\text{and}\qquad p_2(u_1,u_2)=u_1u_2^2-\ell_2u_2,
\end{equation}
we can again give a closed expression for the nonlinear resolvent $(I-h/2\,F)^{-1}$~\cite[p.~133]{MUR}. Second order convergence for the related Peaceman--Rachford discretization is exemplified in Figure~\ref{fig:error}, and the actual pattern formation is illustrated in Figure~\ref{fig:patt}.
}

\begin{figure}[t]
\centering
\subfigure[]{
    \includegraphics[width=0.47\textwidth ]{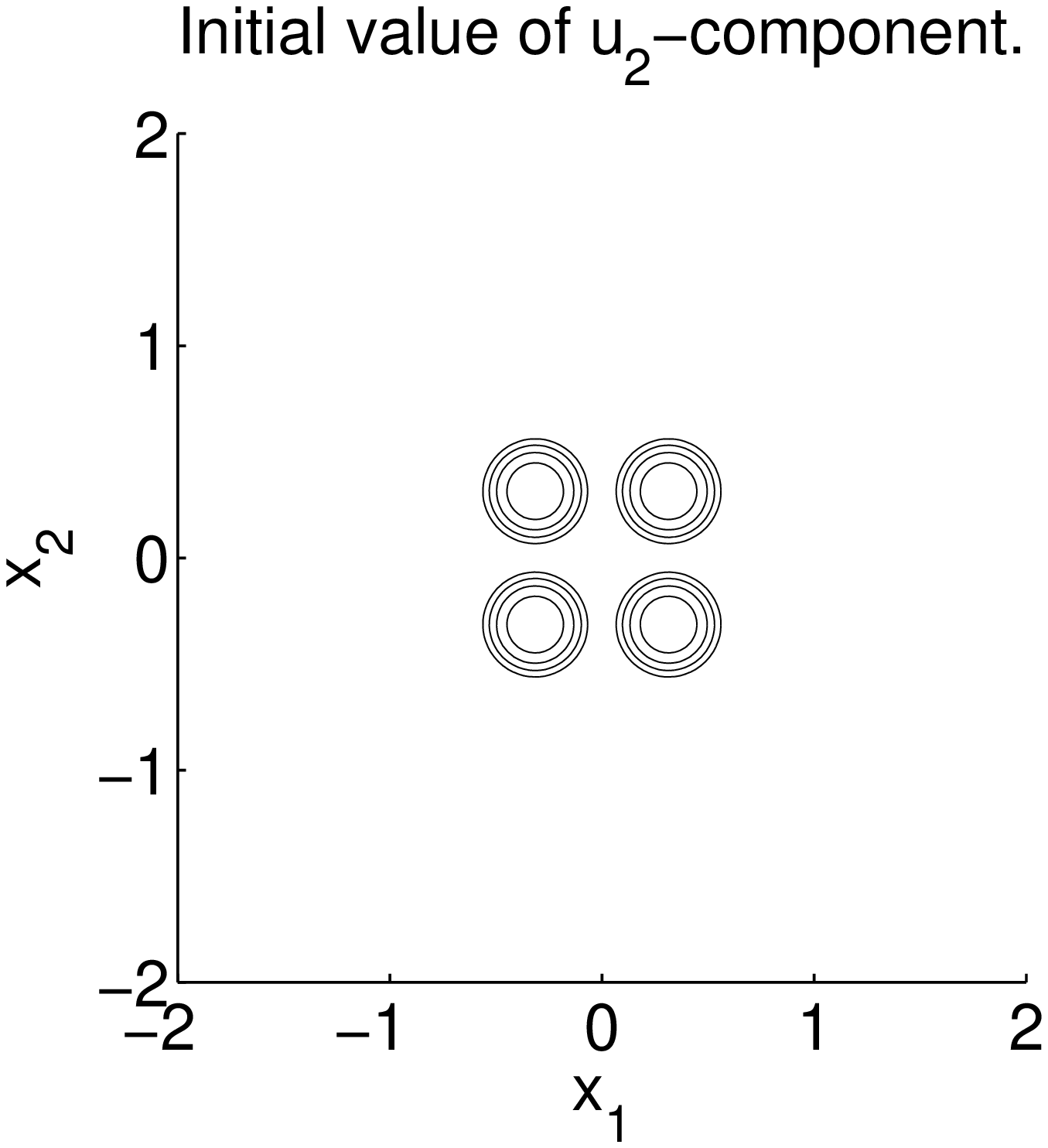}
	\label{fig:gs_t=0}}
\subfigure[]{
    \includegraphics[width=0.47\textwidth ]{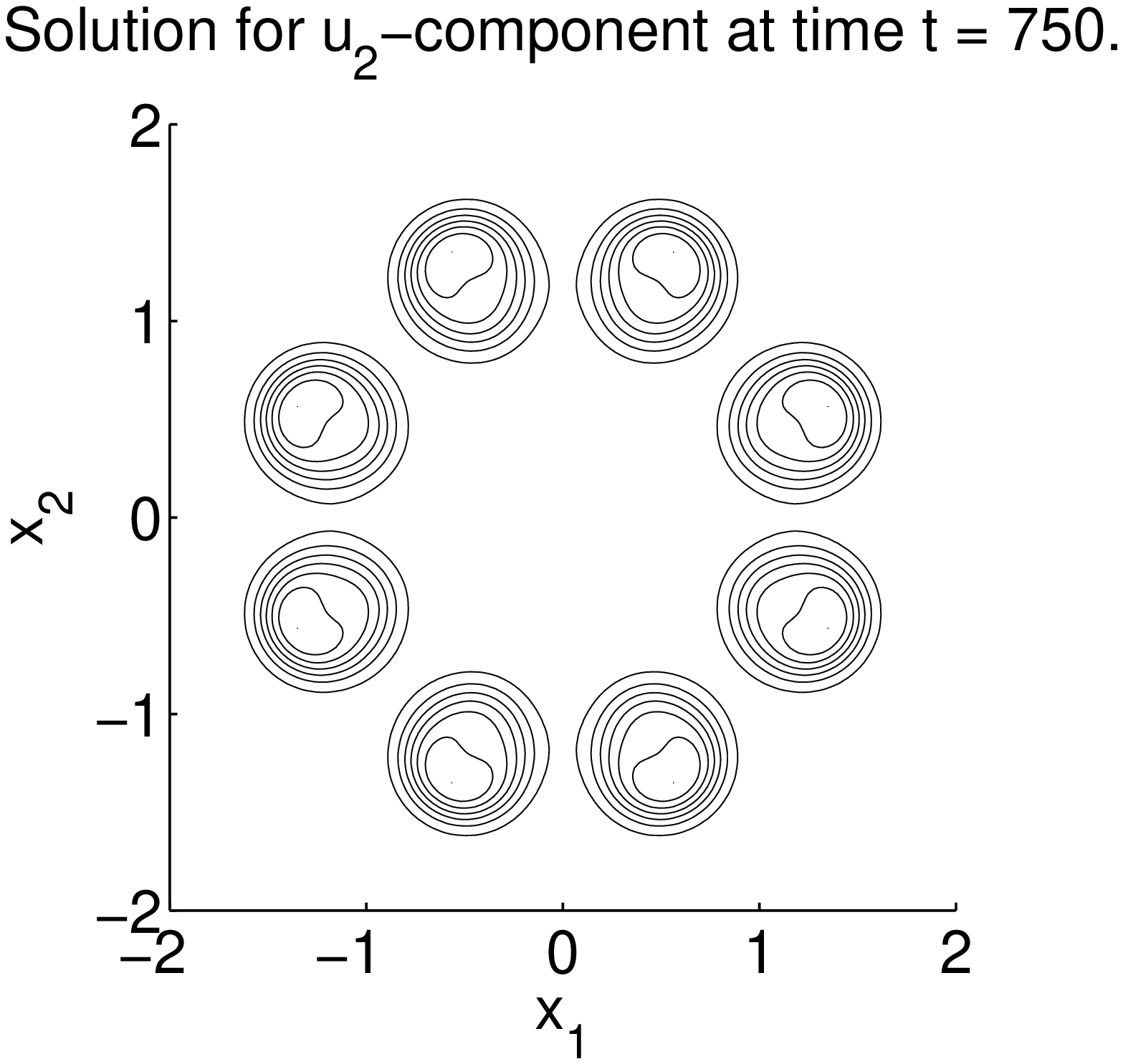}
	\label{fig:gs_t>0}}
\caption{The initial value (see Figure~\ref{fig:error}) of the concentration $u_2$ used for the Gray--Scott solutions can be seen in Figure~\ref{fig:gs_t=0} as contour lines in the $(x_1, x_2)$-plane. The smooth initial spots replicate repeatedly as time increases, this can be seen at time $t = 750$ in Figure~\ref{fig:gs_t>0}. The same parameters as for Figure~\ref{fig:error} are used. For clarity the solutions are displayed for $-2 \leq x_1, x_2\leq 2$ only.}
\label{fig:patt}
\end{figure}

\end{ex}

\end{document}